\documentclass[12pt]{amsart}
\usepackage{amsmath}
\usepackage{amssymb}
\usepackage{amsfonts}
\usepackage{textcomp}
\usepackage{amsthm}
\usepackage{mathrsfs}
\usepackage[latin1]{inputenc}
\usepackage[all]{xy}
\usepackage{hyperref}
\usepackage{url}
\usepackage{graphicx}
\usepackage{tikz-cd}
\usepackage[paperheight=11in,paperwidth=8.5in,left=1in,top=1.25in,right=1in,bottom=1.25in,]{geometry}

\newcommand{\HH}{\mathbb{H}}
\newcommand{\PP}{\mathcal{P}}
\newcommand{\DD}{\mathcal{D}}
\newcommand{\CC}{\mathcal{C}}
\newcommand{\RR}{\mathbb{R}}

\newtheorem{theorem}{Theorem}[section]
\newtheorem{lem}{Lemma}[section]
\newtheorem{prop}{Proposition}[section]
\newtheorem{defi}{Definition}[section]

\begin{document}
\title[Delaunay Triangulations]{Embedded Delaunay triangulations for point clouds of surfaces in $\mathbb{R}^3$}
\author{Franco Vargas Pallete}
\thanks{Research partially supported by NSF grant DMS-1406301 and by the Minerva Research Foundation}
\address{School of Mathematics  \\
 Institute for Advanced Study \\
114 MOS, 1 Einstein Drive \\
Princeton, NJ 08540\\
U.S.A.}
\email{franco@math.ias.edu}
\maketitle

\begin{abstract}
    In the following article we discuss Delaunay triangulations for a point cloud on an embedded surface in $\mathbb{R}^3$. We give sufficient conditions on the point cloud to show that the diagonal switch algorithm finds an embedded Delaunay triangulation.
\end{abstract}

\section{Introduction}

Computational geometry is a topic of interest since, for instance, discusses discrete versions of classic geometric results (see \cite{GLSWI}, \cite{GLSWII} where the authors address the Poincar\'e-Koebe uniformization theorem, or \cite{Luo} for a combinatorial version of the Yamabe flow) and allows algorithmic shape comparison techniques (see for instance \cite{KoehlHass} where genus zero surfaces are compared). Initial conditions that guarantee numerical stability are clearly important in this topic, as well as algorithms to refine and detect good data samples. Some common good condition are the no existence of \textit{very-thin} triangles (each angle should be greater than a predetermined constant) and, more strongly, that every triangle is $\epsilon$-acute (meaning that the angles lie in the interval $(\epsilon,\pi/2-\epsilon)$). This condition in particular implies that two angles sharing their opposite edge add less than $\pi$, which is known as the (strict) \textit{Delaunay condition}, a handy property for a triangulation in computational geometry. Our ability to consistently produce $\epsilon$-acute data will also require the ability to produce (strict) Delaunay data. Delaunay triangulations are also useful in Teichm\"{u}ller theory, since they give a standard decomposition for flat surfaces with singularities (see \cite{MasurSmillie}).

In this article we show that the diagonal switch algorithm (which is known in the literature as extrinsic edge flip, appearing for instance in \cite{Dyer}, \cite{Renka} for this type of problem) for a point cloud of a $C^1$ surface $\Sigma\hookrightarrow\mathbb{R}^3$ finalizes after a finite number of steps realizing a embedded Delaunay triangulation, provided that the point cloud is sufficiently dense and stills reflects the $C^1$ structure. If the point cloud is in generic position (no $4$ coplanar points) then the triangulation is strict Delaunay. We also provide an example where the final triangulation is Delaunay but contains very-thin triangles in a proportion that gets arbitrarily close to $1$.

The article is organized as follows. Section \ref{sec:Plane} states the results for (abstract) flat surfaces, possibly with cone singularities. Section \ref{sec:R3} deals with surfaces in $\mathbb{R}^3$. The algorithm is shown to end by using area and hyperbolic volume (Proposition \ref{prop:alg}), while embeddedness is shown by first proving a local result (Proposition \ref{localDelaunay}) and then a local-to-global description (Theorem \ref{thm:global}). Finally, we discuss the embedded Delaunay example with arbitrary predominance of very-thin triangles.

\textbf{Acknowledgements:} I would like to thank Joel Hass for his comments and encouragement while working in this project. I would like also to thank the Discrete Geometry group at UC Davis, organized by Joel Hass and Patrice Koehl, for being so welcoming and introducing me to the subject.  I would also like to thank Feng Luo for pointing out the question about pointy triangles, as well as to the anonymous referee that pointed out the counterexample to it (Example 1) and further constructive comments.

\section{Delaunay Triangulation in the Plane}\label{sec:Plane}

\begin{defi}
Given flat triangles ABD and BCD with disjoint interiors, we say that the common edge BD satisfies the Delaunay condition (or for short that is a Delaunay edge) if the opposite angles satisfy $\angle BAD + \angle BCD \leq \pi$. If the inequality is strict, we say that the edge satisfies the strict Delaunay condition.
\end{defi}

\begin{defi} We say that a triangulation is Delaunay (resp. strict Delaunay) if for any two adjacent triangles the common edge is Delaunay (resp. strict Delaunay)
\end{defi}

\begin{lem}\label{EdgeDel}
For any ABCD flat quadrilateral at least one of the diagonal divides it in two disjoint triangles satisfying the Delaunay condition, as detailed in the following cases
\begin{enumerate}
	\item If ABCD is convex but not cyclic, the diagonal between the opposite angles that add less than $\pi$ is strict Delaunay.
	\item If ABCD is cyclic (and of course also convex) both diagonal are Delaunay but not strict.
	\item If ABCD is concave, the unique interior edge is strict Delaunay.
\end{enumerate}
\end{lem}

The following lemma discusses existence and uniqueness of Delaunay triangulations following the results of \cite{Rivin}.

\begin{lem}
Given a pair $(P,X)$ where P is a flat polygon with vertices $V_1, V_2, \ldots, V_n$ and $X$ is a set of points $X_1, X_2,\ldots, X_k$ in the interior of $P$, there exists a Delaunay triangular subdivision of $P$ with vertices $V = \lbrace V_1, \ldots, V_n, X_1,\ldots, X_k \rbrace$. Moreover, any two such Delaunay triangular subdivisions are equivalent by interchanging diagonals (diagonal switch) in a cyclic polygon. 
\end{lem}
\begin{proof}
As in \cite{Rivin}, take the plane as the boundary of the upper half-space model for $\HH^3$. For each triangle in the triangulation, consider the ideal triangle in $\HH^3$ span by its vertices and the ideal tetrahedra they span with $\infty$. Then an edge satisfies the Delaunay condition if and only if the ideal triangles have on top a convex dihedral angle, since the dihedral angle is equal to the sum of opposite angles. 

Given any triangular subdivision of $P$ we can take an edge that does not satisfy the Delaunay condition and proceed with a diagonal switch. Note that by the analysis in Lemma \ref{EdgeDel} the diagonal switch is always possible. This not only will give us a new Delaunay edge but also strictly increases the sum of volume of the tetrahedra. Since there are finitely many triangular subdivisions of $P$, this algorithm will end in a finite number of steps at a Delaunay triangulation. 

Notice then that a Delaunay triangulation is determined by $C(V, P)$, the convex hull of $V$ in $\mathbb{H}^3$ relative to $P$. This set is the minimal set containing $V$ so that if for $x,y\in C(V, P)$ the geodesic segment $\gamma$ joining them in $\mathbb{H}^3$ projects in $\mathbb{R}^2\subset\partial\mathbb{H}^3$ as a subset of $P$, then $\gamma \subset C(V,P)$. This relative convex hull is an ideal polyhedra and the Delaunay triangulation is a subdivision of the projection of the non-vertical faces. Different Delaunay triangulations come from having to subdivide non-triangular faces.
\end{proof}

The relative convex hull description is useful since it can help to understand how to triangulation evolve as we take sequences of point sets. This will be used on our final example.

The construction of Delaunay triangulations can be generalized to triangulations of flat surfaces with cone singularities, where the vertex set and the singularity set coincide. A proof of the following lemma can be found for instance in \cite{MasurSmillie} using Voronoi cells. 

\begin{lem} For every triangle in a 2 dimensional Delaunay triangulation of a flat manifold with cone singularities, the interior of its circumscribed disk does not contain any vertex. If a triangle $ABC$ is such that the closed circumscribed disk contains only $A, B, C$, then it is part of any Delaunay triangulation.
\end{lem}

Another property of a Delaunay switch in the Euclidean plane (from a triangulation $T$ to a triangulation $T'$ is that triangles do not get more \textit{pointy} More precisely, the smallest angle in $T$ is greater than or equal to the smallest angle in $T'$. Since the only angles that change are the ones around the diagonal switch, the non-decreasing property of the smallest angle can be seen after the following two observations:

\begin{itemize}
    \item The smallest angle of the 2 triangles adjacent to a non-strict Delaunay edge can only be adjacent to the diagonal.
    \\ \noindent
    Indeed, If $\alpha, \beta$ are the angles facing the non-strict Delaunay diagonal, then $\alpha+\beta\geq \pi$. But since the sum of angles in a triangle is $\pi$, each angle adjacent to the diagonal is strictly less than $\alpha$ or $\beta$.
    \item For each angle adjacent to the diagonal in $T"$ there is a angle adjacent to the diagonal in $T$ that is not bigger. If the new diagonal is strict Delaunay then the comparison angles are strictly bigger. Angles facing the diagonal in $T'$ are sum of angles in $T$.\\ \noindent
    Notice that for $4$ points $A, B, C, D$ (with no $3$ colinear points), $D$ is not in the interior of the circumcircle of $A, B, C$ if and only if $\angle ADB \leq \angle  ACB$, with equality if and only if $D$ belongs to the incircle. But this can be rephrase as changing into the Delaunay diagonal $AC$, where the mentioned angles are the adjacent angles to the diagonals.
\end{itemize}

We will in Examples 1 and 2 that such analogy does not quite work for surfaces in $\mathbb{R}^3$.

\section{Delaunay Triangulations for embedded surfaces in $\mathbb{R}^3$}\label{sec:R3}

Let us observe first that a diagonal switch changes one non-Delaunay diagonal for a strict Delaunay diagonal. Namely, if for $A, B, C, D \in \mathbb{R}^3$ the common edge $BD$ of the adjacent triangles $ABD$ and $BCD$ is non-Delaunay ($\angle BAD + \angle BCD > \pi$) then the common edge $AC$ of the adjacent triangles $ABC$ and $ACD$ is strict Delaunay ($\angle ABC + \angle ADC < \pi$). Indeed, by angular triangle inequality $\angle BAC + \angle CAD \geq \angle BAD$ and $\angle BCA + \angle ACD \geq \angle BCD$, so then

\[\angle ABC + \angle ADC = \pi -(\angle BAC + \angle BCA) + \pi -(\angle CAD + \angle ACD) \leq  2\pi - \angle BAD - \angle BCD < \pi\]

Notice that if $A, B, C, D$ are not coplanar then the angular triangle inequality is strict, so then doing a Delaunay switch when $\angle BAD + \angle BCD = \pi$ still implies that $AC$ is strict Delaunay. This is different from the flat case, where any diagonal of a cyclic quadrilateral gives an angle sum of $\pi$.

One point for us to address is to show that the process ends in a finite number of steps, producing a Delaunay triangulation. Inspired by Rivin's argument in [\cite{Rivin}, Section 6] using hyperbolic volumes, we make use of the following lemma (also proven in [\cite{Dyer},Theorem 6.5]).

\begin{lem}\label{lem:area}
Let $A, B, C, D \in \mathbb{R}^3$ such that $\angle BAD + \angle BCD \geq \pi$ and denote by $|\cdot|$ the area of a flat region in $\mathbb{R}^3$. Then $|ABC|+|ADC| \leq |ABD|+|BCD|$, with equality if and only if the 4 points $A, B , C, D$ are coplanar and form a convex quadrilateral.
\end{lem}
\begin{proof}

Assume that the 4 points are not coplanar since the inequality is obvious otherwise. Denote $\alpha = \angle ABC, \beta = \angle ADC$ (see Figure \ref{fig:ABCD}). Note that $\alpha + \beta < \pi$. If we also denote by $a= \ell(AB), b=\ell(BC), c=\ell(CD), d=\ell(DA)$ then by cosine law we have
\[\ell (AC) = a^2+b^2 - 2ab\cos(\alpha) = c^2+d^2 - 2cd\cos(\beta)\]

\begin{figure}
\centering
\tikzset{every picture/.style={line width=0.75pt}} 

\begin{tikzpicture}[x=0.75pt,y=0.75pt,yscale=-1,xscale=1]

\draw  (342.5,195.8) -- (324,49) ;

\draw  [dash pattern={on 4.5pt off 4.5pt}]  (424,149) -- (250.5,144.8) ;

\draw    (342.5,195.8) -- (424,149) ;

\draw    (324,49) -- (250.5,144.8) ;

\draw    (250.5,144.8) -- (342.5,195.8) ;

\draw    (324,49) -- (424,149) ;

\draw  [draw opacity=0][fill={rgb, 255:red, 74; green, 144; blue, 226 }  ,fill opacity=0.22 ] (258.68,133.96) .. controls (264.76,136.02) and (268.89,140.11) .. (268.86,144.79) .. controls (268.84,147.81) and (267.09,150.57) .. (264.2,152.67) -- (249.83,144.67) -- cycle ; \draw  [color={rgb, 255:red, 14; green, 110; blue, 223 }  ,draw opacity=1 ] (258.68,133.96) .. controls (264.76,136.02) and (268.89,140.11) .. (268.86,144.79) .. controls (268.84,147.81) and (267.09,150.57) .. (264.2,152.67) ;
\draw  [draw opacity=0][fill={rgb, 255:red, 74; green, 144; blue, 226 }  ,fill opacity=0.22 ] (410.53,156.73) .. controls (408.76,155.19) and (407.61,153.28) .. (407.31,151.11) .. controls (406.73,146.87) and (409.56,142.7) .. (414.29,139.92) -- (425.07,148.67) -- cycle ; \draw  [color={rgb, 255:red, 74; green, 144; blue, 226 }  ,draw opacity=1 ] (410.53,156.73) .. controls (408.76,155.19) and (407.61,153.28) .. (407.31,151.11) .. controls (406.73,146.87) and (409.56,142.7) .. (414.29,139.92) ;

\draw (324,39) node  [align=left] {$A$};
\draw (281,92) node  [align=left] {$a$};
\draw (239,145) node  [align=left] {$B$};
\draw (288,175) node  [align=left] {$b$};
\draw (342,206) node  [align=left] {$C$};
\draw (386,178) node  [align=left] {$c$};
\draw (434,150) node  [align=left] {$D$};
\draw (383,95) node  [align=left] {$d$};
\draw (278,138) node [color={rgb, 255:red, 74; green, 144; blue, 226 }  ,opacity=1 ] [align=left] {$\displaystyle \alpha $};
\draw (397,140) node [color={rgb, 255:red, 74; green, 144; blue, 226 }  ,opacity=1 ] [align=left] {$\displaystyle \beta $};

\end{tikzpicture}
\caption{$ABCD$}\label{fig:ABCD}
\end{figure}

Now, fixing the side lengths consider the figure as a function of the angle $\alpha$. Then from implicit differentiation we have that $2\sin(\alpha)ab = 2\partial_\alpha \beta .cd\sin(\beta)$. Since the total area $A$ is equal to $\frac12(ab\sin(\alpha) + cd\sin(\beta))$  then we have the formula

\begin{equation}\label{eq:darea}
    \begin{split}
        \partial_\alpha A &= \frac12(ab\cos(\alpha) + cd.\partial_\alpha \beta.\cos(\beta)) = \frac{ab}{2}\left(\cos(\alpha) + \frac{\cos(\beta).\sin(\alpha)}{\sin(\beta)}\right)\\
        &= \frac{ab\sin(\alpha+\beta)}{2\sin(\beta)} > 0
    \end{split}
\end{equation}

Now, if we assemble isometrically the triangles $ABC$ and $ACD$ by $AC$ we obtain a quadrilateral with the same side lengths as if we assemble $ABD$ and $BCD$ by $BD$. By the angular triangle inequality the angle at $B$ for $ABC+ACD$ is smaller than the angle for $ABD+BCD$, while for both quadrilaterals the sum of angles at $B$ and $D$ is less than $\pi$. Because of (\ref{eq:darea}) the bigger area corresponds to the bigger angle, so then $|ABC|+|ADC| < |ABD|+|BCD|$

\end{proof}

\begin{defi}
Let $\PP\subset \mathbb{R}^3$ be a set of finite points and $\Sigma$ a surface with a triangulation $T$. We say that a map $f:(\Sigma, T) \rightarrow \mathbb{R}^3$ is a realization of $(\Sigma, T)$ into $\PP$ if it is linear in each triangle and bijects the vertex set $V$ to $\PP$. If the set $\PP$ is taken as a subset of an embedded surface $\Sigma$ we will take $V=\PP$ and $f|_{V}$ as the inclusion map. We say in this case that $\PP$ is a point cloud for $\Sigma$.
\end{defi}

Notice that a realization $f$ of $(\Sigma, T)$ into $\PP$ is determined once a bijection between $V$ and $\PP$ is fixed. Then we can say a realization of $(\Sigma,T)$ into $\PP$ is Delaunay or strict Delaunay two adjacent images triangles are Delaunay or strict Delaunay. We can also perform diagonal switches by changing the triangulation $T$ but preserving the vertex set $V$ and its bijection to $\PP$. Finally, we can calculate the area of a realization ($A(T)$ for short) as the sum of the areas for each flat triangle and the hyperbolic volume of the realization ($vol(T)$) as the sum of hyperbolic volumes defined for each flat triangle. This serves for the following proposition.

\begin{prop}\label{prop:alg}
If there is a realization of $(\Sigma,T)$ into $\PP$, then we can change the triangulation while keeping the vertex set and bijection with $\PP$ so that the realization $(\Sigma, T_0)$ is Delaunay. If no 4 points of $\PP$ are coplanar, then there is a strict Delaunay realization.
\end{prop}
\begin{proof}
This follows after observing that a diagonal switch strictly decreases $(A(T),-vol(T))$. Indeed, if the diagonal switch happens among 4 non-coplanar points then the area sum strictly decreases. If the diagonal switch happens for 4 coplanar point the area sum decreases or stays the same, but the hyperbolic volume increases. Given that there are finitely many realizations, the algorithm ends in finite time. The later claim follows from Lemma \ref{lem:area} given that now even when the sum of angle is $\pi$ the diagonal switch strictly decreases area while changing into a Delaunay edge.
\end{proof}

Given the improving nature of this lemma, we will simply say that $\PP$ realizes $\Sigma$, while the triangulation associated will be stated by context.

Another difference with diagonal switches for flat surfaces is that in $\mathbb{R}^3$ a embedded triangulation could end self-intersecting. For instance, take a tetrahedra $ABCD$ where $BD$ is non-Delaunay. After a diagonal switch the figure degenerates into the union of the triangles $ABC, ACD$ with multiplicity $2$. This compels us to request extra properties for a point cloud in order to preserve embeddedness. 

\begin{defi}\label{def:dense}
We say that a point cloud $\PP$ of a surface $\Sigma$ is $\delta$-dense if any ball of radius $\delta$ in $\Sigma$ contains at least one point from $\PP$.
\end{defi}

\begin{defi}\label{def:aflat}
We say that a point cloud $\PP$ of a surface $\Sigma$ is $(\theta,r)$-almost flat if for $r>0$ we have that for any given point $p\in\PP$ there is a plane $L$ such that any $3$ points of $\PP$ in the ball of radius $r$ at $p$ form a plane with angle less that $\theta$ with $L$.
\end{defi}

Note that $\Sigma$ being $C^1$-regular is enough to guarantee that there is a point cloud $\PP$ satisfying Definitions \ref{def:dense} and \ref{def:aflat}. A point cloud will also satisfy these conditions if the error in measure is $C^1$-small.

In the following proposition we see an almost flat set of points stays embedded after diagonal switches.

\begin{prop}\label{localDelaunay} Let $\PP$ be a point cloud for a surface $\Sigma$ and $L$ be a plane in $\mathbb{R}^3$. If every plane made from 3 points of $\PP$ has angle less than $\pi/8$ with $L$, then there is a Delaunay realization of $\Sigma$ into $\PP$ that is an embedding in $\mathbb{R}^3$. This realization is relative to $\partial \Sigma$. 
\end{prop}
\begin{proof}
Notice that the angle condition implies that the orthogonal projection from $\PP$ to $L$ is $1$-to-$1$. Then triangulate the image $\PP'$ of $\PP$ in $L$ and take the corresponding triangulation for $\PP$. Clearly this triangulation also projects $1$-to-$1$ into $L$, so it is embedded. We will prove that after every diagonal switch, the projection of the new triangulation of $\PP$ still projects $1$-to-$1$ into $L$. This clearly proves embeddedness.

The only possibility for a diagonal switch from $\PP$ to break the $1$-to-$1$ projection property is that in $\PP'$ the corresponding quadrilateral is concave. Take then $4$ points $A, B, C, D$ in $\PP$ such that their projections $A', B', C', D'$ in $L$ form a concave quadrilateral. Assume without loss of generality that $D'$ is contained in the triangle $A'B'C'$. Given that any plane made by points in $\PP$ has an angle less than $\pi/8$ with $L$, then any two different planes made by points in $\PP$ have an angle less than $\pi/4$. Then the line $DD'$ (orthogonal to $A'B'C'$) intersects the plane $L_1$ spanned by $A,B,C$ at an interior point $D_1$ of $ABC$ forming an angle $\alpha$ greater or equal than $\pi/2-\pi/8=3\pi/8$. The angle $\beta$ between $DAB$ and $ABC$ facing $DD_1$ should be either less than $\pi/4$ or greater than $3\pi/4$. Let then $D_0$ be the orthogonal projection of $D$ to $AB$. Since the geometric locus of points $E\in L_1$ such that $DE$ has a given angle with $L_1$ are concentric circles and the angle decreases as the circle increases, then $\beta < \pi/4$. Indeed, if $\beta > 3\pi/4$, the geometric locus of points $E\in L_1$ such that the angle between $DE$ with $L_1$ is at least $\pi-\beta<\pi/4$ is a disk tangent to $AB$ at $D_0$ on the opposite side to $C$. This is a contradiction since the angle $\alpha$ between $DD_1$ and $L_1$ is at least $3\pi/8>\pi/4$ but is at the same side as $C$.

\begin{figure}
    \centering

\tikzset{every picture/.style={line width=0.75pt}} 

\begin{tikzpicture}[x=0.75pt,y=0.75pt,yscale=-1,xscale=1]

\draw  [dash pattern={on 0.84pt off 2.51pt}]  (359,70) -- (361.5,239.8) ;
\draw [shift={(361.5,239.8)}, rotate = 89.16] [color={rgb, 255:red, 0; green, 0; blue, 0 }  ][fill={rgb, 255:red, 0; green, 0; blue, 0 }  ][line width=0.75]      (0, 0) circle [x radius= 3.35, y radius= 3.35]   ;
\draw [shift={(359,70)}, rotate = 89.16] [color={rgb, 255:red, 0; green, 0; blue, 0 }  ][fill={rgb, 255:red, 0; green, 0; blue, 0 }  ][line width=0.75]      (0, 0) circle [x radius= 3.35, y radius= 3.35]   ;
\draw   (206,268) -- (368.5,268.8) -- (449.5,192.8) -- (205.5,267.8) -- cycle ;
\draw  [color={rgb, 255:red, 74; green, 144; blue, 226 }  ,draw opacity=1 ] (209,125) -- (367.5,165.8) -- (443.5,110.8) -- cycle ;
\draw  [draw opacity=0][fill={rgb, 255:red, 74; green, 144; blue, 226 }  ,fill opacity=0.22 ] (344.38,134.52) .. controls (348.01,131.86) and (353.51,130.18) .. (359.65,130.22) .. controls (359.76,130.22) and (359.88,130.22) .. (359.99,130.23) -- (359.58,142.22) -- cycle ; \draw  [color={rgb, 255:red, 74; green, 144; blue, 226 }  ,draw opacity=1 ] (344.38,134.52) .. controls (348.01,131.86) and (353.51,130.18) .. (359.65,130.22) .. controls (359.76,130.22) and (359.88,130.22) .. (359.99,130.23) ;
\draw [color={rgb, 255:red, 208; green, 2; blue, 27 }  ,draw opacity=1 ] [dash pattern={on 4.5pt off 4.5pt}]  (358.67,70) -- (306.5,150.8) ;

\draw  [draw opacity=0][fill={rgb, 255:red, 74; green, 144; blue, 226 }  ,fill opacity=0.22 ] (316.94,134.1) .. controls (319.29,136.16) and (321.16,139) .. (322.28,142.31) -- (307.83,149.66) -- cycle ; \draw  [color={rgb, 255:red, 74; green, 144; blue, 226 }  ,draw opacity=1 ] (316.94,134.1) .. controls (319.29,136.16) and (321.16,139) .. (322.28,142.31) ;
\draw    (307.6,150.6) -- (329.65,139.48) ;

\draw    (338.36,131.58) -- (360.15,142.38) ;

\draw (372.67,66.67) node  [align=left] {$\displaystyle D$};
\draw (198.67,124.33) node [color={rgb, 255:red, 74; green, 144; blue, 226 }  ,opacity=1 ] [align=left] {$\displaystyle A$};
\draw (377,169.67) node [color={rgb, 255:red, 74; green, 144; blue, 226 }  ,opacity=1 ] [align=left] {$\displaystyle B$};
\draw (454,109) node [color={rgb, 255:red, 74; green, 144; blue, 226 }  ,opacity=1 ] [align=left] {$\displaystyle C$};
\draw (193.67,269) node [color={rgb, 255:red, 0; green, 0; blue, 0 }  ,opacity=1 ] [align=left] {$\displaystyle A'$};
\draw (381.67,268) node [color={rgb, 255:red, 0; green, 0; blue, 0 }  ,opacity=1 ] [align=left] {$\displaystyle B'$};
\draw (460.33,189.67) node  [align=left] {$\displaystyle C'$};
\draw (350.2,122.6) node [scale=0.9,color={rgb, 255:red, 74; green, 144; blue, 226 }  ,opacity=1 ]  {$\alpha $};
\draw (120,150) node [scale=0.8] [align=left] {$ $};
\draw (120,131) node [scale=0.8] [align=left] {$ $};
\draw (328.2,130.6) node [scale=0.9,color={rgb, 255:red, 74; green, 144; blue, 226 }  ,opacity=1 ]  {$\beta $};
\draw (372,247) node  [align=left] {$\displaystyle D'$};
\draw (304,163) node [scale=0.9] [align=left] {$\displaystyle D_{0}$};
\draw (371,141.75) node [scale=0.9] [align=left] {$\displaystyle D_{1}$};

\end{tikzpicture}

    \caption{}
    \label{fig:projABC}
\end{figure}

Now let us prove that $\angle DAB \leq \angle CAB$.  Because of the previous paragraph, $D$ belongs to the convex polyhedra with vertices $ABCEF$ (see Figure \ref{fig:ABCEF}), where $E$ and $F$ are found by intersecting the planes containing an edge of $ABC$ and that make a $\pi/4$ angle with $ABC$. Given the symmetries of such planes, $E$ and $F$ are reflection of one another with respect to $ABC$ and their midpoint is the incenter of $ABC$, denoted by $I$.  The result will follow after proving that the angle $\angle DAB$ is maximized (while varying $D$) when $D$ is one of the vertices of $ABCEF$, and then verifying the inequality $\angle DAB \leq \angle CAB$ for $D=A,B,C,E,F$. 

\begin{figure}
    \centering

\tikzset{every picture/.style={line width=0.75pt}} 

\begin{tikzpicture}[x=0.75pt,y=0.75pt,yscale=-1,xscale=1]

\draw   (230.48,106) -- (419.5,187.8) -- (135.54,187.8) -- cycle ;
\draw [line width=0.75]  [dash pattern={on 0.84pt off 2.51pt}]  (250,71) -- (249.5,246.8) ;
\draw [shift={(249.5,246.8)}, rotate = 90.16] [color={rgb, 255:red, 0; green, 0; blue, 0 }  ][fill={rgb, 255:red, 0; green, 0; blue, 0 }  ][line width=0.75]      (0, 0) circle [x radius= 3.35, y radius= 3.35]   ;
\draw [shift={(250,71)}, rotate = 90.16] [color={rgb, 255:red, 0; green, 0; blue, 0 }  ][fill={rgb, 255:red, 0; green, 0; blue, 0 }  ][line width=0.75]      (0, 0) circle [x radius= 3.35, y radius= 3.35]   ;
\draw [color={rgb, 255:red, 208; green, 2; blue, 27 }  ,draw opacity=1 ]   (250,71) -- (135.86,187.54) ;

\draw [color={rgb, 255:red, 208; green, 2; blue, 27 }  ,draw opacity=1 ] [dash pattern={on 4.5pt off 4.5pt}]  (250.13,70.88) -- (230.63,105.38) ;

\draw [color={rgb, 255:red, 208; green, 2; blue, 27 }  ,draw opacity=1 ]   (249.89,71.42) -- (419,188.04) ;

\draw [color={rgb, 255:red, 208; green, 2; blue, 27 }  ,draw opacity=1 ]   (135.14,187.86) -- (249.57,246.69) ;

\draw [color={rgb, 255:red, 208; green, 2; blue, 27 }  ,draw opacity=1 ]   (249.57,246.97) -- (419,187.54) ;

\draw [color={rgb, 255:red, 208; green, 2; blue, 27 }  ,draw opacity=1 ] [dash pattern={on 4.5pt off 4.5pt}]  (230.64,106.56) -- (249.5,246.8) ;

\draw (125,186) node  [align=left] {$\displaystyle A$};
\draw (428,186) node  [align=left] {$\displaystyle B$};
\draw (241,100) node  [align=left] {$\displaystyle C$};
\draw (262,68) node  [align=left] {$\displaystyle E$};
\draw (259,235) node  [align=left] {$\displaystyle F$};

\draw (262,154) node  [align=left] {$\displaystyle I$};

\end{tikzpicture}

    \caption{}
    \label{fig:ABCEF}
\end{figure}

Given a segment $AB$, the geometric locus for $\angle DAB = const.$ are one end cones at $A$ with axis $AB$, without the cone vertex $A$. Because of these level sets a ray $\overrightarrow{AD}$ maximizing $\angle DAB$ (for $D\in ABCEF$) has to lie on a face containing $A$. If $\overrightarrow{AD}$ does not contain neither $B,C,D$ nor $E$, then the face of $ABCEF$ that contains $\overrightarrow{AD}$ has to be tangent (along $\overrightarrow{AD}$) to the cone defined by rotating $\overrightarrow{AD}$ using $AB$ as an axis. This implies that $DAB$ and the face of $ABCEF$ containing $\overrightarrow{AD}$ are orthogonal, but by construction each of them make as an angle less than $\pi/4$ with $ABC$, which is a contradiction.

Since the inequality $\angle DAB \leq \angle CAB$ follows easily when $D$ is either $B$ or $C$, we need to only check for $E$ or $F$, and those cases are analogous. Name $E'$ the projection of $E$ to $AB$ (see Figure \ref{fig:tan}). Then $\tan(\angle EAB) = \frac{\ell(EE')}{\ell(AE')} = \sec(\pi/4).\frac{\ell(IE')}{\ell(AE')} =  \sqrt{2} \tan (\angle CAB/2)$, so by using the double angle formula follows that $\angle EAB \leq \angle CAB$ as an easy exercise. By an analogous procedure we have that $\angle DIJ \leq \angle KIJ$, for distinct $I,J,K\in\lbrace A, B, C \rbrace$.

\begin{figure}
    \centering

\tikzset{every picture/.style={line width=0.75pt}} 

\begin{tikzpicture}[x=0.75pt,y=0.75pt,yscale=-1,xscale=1]

\draw   (230.48,106) -- (419.5,187.8) -- (135.54,187.8) -- cycle ;
\draw [color={rgb, 255:red, 208; green, 2; blue, 27 }  ,draw opacity=1 ]   (250,71) -- (135.86,187.54) ;

\draw [color={rgb, 255:red, 208; green, 2; blue, 27 }  ,draw opacity=1 ] [dash pattern={on 4.5pt off 4.5pt}]  (250.13,70.88) -- (230.63,105.38) ;

\draw [color={rgb, 255:red, 208; green, 2; blue, 27 }  ,draw opacity=1 ]   (249.89,71.42) -- (419,188.04) ;

\draw  [dash pattern={on 0.84pt off 2.51pt}]  (249.89,71.56) -- (249.8,151.66) ;

\draw  [dash pattern={on 4.5pt off 4.5pt}]  (135.33,187.73) -- (250.13,148.45) ;

\draw  [draw opacity=0][fill={rgb, 255:red, 74; green, 144; blue, 226 }  ,fill opacity=0.22 ] (153.91,172.37) .. controls (164.94,175.55) and (172.49,181.38) .. (173.03,188.07) -- (135.07,188.66) -- cycle ; \draw  [color={rgb, 255:red, 74; green, 144; blue, 226 }  ,draw opacity=1 ] (153.91,172.37) .. controls (164.94,175.55) and (172.49,181.38) .. (173.03,188.07) ;
\draw  [draw opacity=0][fill={rgb, 255:red, 208; green, 2; blue, 27 }  ,fill opacity=0.13 ] (167.11,155.58) .. controls (187.05,161.85) and (200.5,173.8) .. (200.42,187.42) .. controls (200.42,187.49) and (200.42,187.55) .. (200.42,187.62) -- (135.75,187.01) -- cycle ; \draw  [color={rgb, 255:red, 208; green, 2; blue, 27 }  ,draw opacity=1 ] (167.11,155.58) .. controls (187.05,161.85) and (200.5,173.8) .. (200.42,187.42) .. controls (200.42,187.49) and (200.42,187.55) .. (200.42,187.62) ;
\draw  [dash pattern={on 4.5pt off 4.5pt}]  (249.8,148.3) -- (280.2,187.7) ;

\draw  [dash pattern={on 0.84pt off 2.51pt}]  (249.85,71.66) -- (280.33,187.44) ;

\draw [color={rgb, 255:red, 189; green, 16; blue, 224 }  ,draw opacity=1 ]   (250.11,138.33) -- (240.28,140.97) -- (240.28,151.97) ;

\draw [color={rgb, 255:red, 189; green, 16; blue, 224 }  ,draw opacity=1 ]   (278.22,178.88) -- (267.78,178.66) -- (270,187.77) ;

\draw [color={rgb, 255:red, 189; green, 16; blue, 224 }  ,draw opacity=1 ]   (275.78,181.99) -- (257.56,181.99) -- (262,187.77) ;

\draw  [draw opacity=0][fill={rgb, 255:red, 74; green, 144; blue, 226 }  ,fill opacity=0.22 ] (269.94,174.87) .. controls (271.98,173.99) and (274.25,173.37) .. (276.67,173.04) -- (280.53,187.64) -- cycle ; \draw  [color={rgb, 255:red, 74; green, 144; blue, 226 }  ,draw opacity=1 ] (269.94,174.87) .. controls (271.98,173.99) and (274.25,173.37) .. (276.67,173.04) ;
\draw [color={rgb, 255:red, 189; green, 16; blue, 224 }  ,draw opacity=1 ]   (249.7,138.5) -- (256.08,146.97) -- (256.08,155.97) ;

\draw (125,186) node  [align=left] {$\displaystyle A$};
\draw (428,186) node  [align=left] {$\displaystyle B$};
\draw (241,100) node  [align=left] {$\displaystyle C$};
\draw (262,68) node  [align=left] {$\displaystyle E$};
\draw (167.79,167.64) node [scale=0.9,color={rgb, 255:red, 74; green, 144; blue, 226 }  ,opacity=1 ,xslant=0.31] [align=left] {$\displaystyle \alpha $};
\draw (198.92,155.83) node [color={rgb, 255:red, 208; green, 2; blue, 27 }  ,opacity=1 ] [align=left] {$\displaystyle \beta $};
\draw (180.04,178.39) node [scale=0.9,color={rgb, 255:red, 74; green, 144; blue, 226 }  ,opacity=1 ,xslant=0.31] [align=left] {$\displaystyle \alpha $};
\draw (297.33,172.67) node [scale=1,color={rgb, 255:red, 74; green, 144; blue, 226 }  ,opacity=1 ] [align=left] {$\displaystyle \pi /4$};
\draw (283,198) node  [align=left] {$\displaystyle E'$};
\draw (261,138) node  [align=left] {$\displaystyle I$};

\end{tikzpicture}

    \caption{}
    \label{fig:tan}
\end{figure}

Recall now that in order to do a diagonal switch the two opposite angles to the erased diagonal need to add more than $\pi$. But because of the previous paragraph, each of them adds to less than two of the interior angles of $ABC$. Since this sum is strictly less than $\pi$, we have a contradiction.
\end{proof}

Now let us produce a global Delaunay triangulation by applying Proposition \ref{localDelaunay} locally and showing a control in edge length to complete the local-to-glocal argument.

\begin{theorem}\label{thm:global}
Let $\PP$ be a point cloud of a surface $\Sigma$ that is $(\pi/16,r)$-almost flat for some $r>0$. There is $\delta>0$ sufficiently small with respect to $r$ such that if $\PP$ is $\delta$-dense then there is a Delaunay triangulation of $\PP$ that is an embedding into $\RR^3$. Such Delaunay triangulation is obtained by diagonal switches. 
\end{theorem}
\begin{proof}
Assume for now that we start with a triangulation where all edges have length less than a given small constant $\epsilon\ll 1$. With some of the setup we will develop, this step will be easier to explain at the end.

\textit{Step 1. Local modification}\\\noindent
Our goal is to setup a local Delaunay triangulation in a given disk neighbourhood of controlled size. Without loss of generality let us assume that every disk of radius $2r$ in $\Sigma$ is $(\pi/16,r)$ almost flat. Assume also $\epsilon$ small enough so that $\epsilon < \frac{r}{1000}$. Given a closed disk $\DD(r)$ of radius $r$, take all the triangles that intersect $\DD$ and denote by $T_0$ the piecewise flat surface obtained by the union of those triangles. Let $T$ be the union of $T_0$ with all the triangulated planar components of its complement $T_0^C$ (i.e. with $0$ genus) and $V$ its vertices set. $T$ only ommits the component of $T_0^C$ that contains $\DD_{2r}^C$. One can observe that $T$ is a topological disk, $\DD_r \subseteq T\subseteq\DD_{r+\epsilon}$ and $\partial T \subset \DD_{r+\epsilon}\setminus \DD_r$. Let us then do the diagonal switches in $T$ until $(T,\partial T)$ is Delaunay. By Proposition \ref{localDelaunay} this new triangulation is an embedding, where the projection to the reference plane is 1-to-1.

\textit{Step 2. Bounding the largest edge}\\\noindent
Let us analyze how we could obtain an edge of length at least $\epsilon$. We will prove that any new edge has length at most $2\epsilon$, and any new edge at least one vertex in $\mathcal{D}_{r-3\epsilon}$ has length at most $\epsilon$. The constrain will come from assuming that there is a long edge and then produce a disk of area $A(\epsilon)> \pi\delta^2$ without vertices, which contradicts the $\delta$-density. Let then $\overline{AB}$ be the longest edge of the Delaunay triangulation of $(T, \partial T)$. Since the switches were done relative to $\partial T$ where length is already bounded by $\epsilon$, let us assume that $\overline{AB}$ is an interior edge.

For this step, our claim is that we can select $\delta>0$ small enough so that:

\begin{enumerate}
\item\label{1stclaim} If $A, B$ belong to $\DD_{r-\epsilon}$ then $\ell(\overline{AB})\leq \epsilon$
\item\label{2ndclaim} If either $A,B$ belong to $T\setminus \DD_{r-\epsilon}\subseteq\DD_{r+\epsilon}\setminus\DD_{r-\epsilon}$ then $\ell(\overline{AB}) \leq 2\epsilon$
\end{enumerate}
To prove (\ref{1stclaim}), there are vertices $C,D$ such that $CAB, DAB$ are part of the triangulation and satisfy the Delaunay condition. Assume without lost of generality that $\angle ACB \leq \angle ADB$, and since their sum is less than $\pi$, that $\angle ACB \leq \pi/2$. The disk $\Delta\subset (T,\partial T)$ circumscribed to $ABC$ does not contain vertices from $T$ and has radius at least $\ell(\overline{AB})/2 > \epsilon/2$, so area at least $\pi\ell(\overline{AB})^2/4 > \pi\epsilon^2/4$. $V$ is a point cloud $(\pi/8,r)$ almost flat with respect to the plane $P$ defined by $ABC$, so by Proposiion \ref{localDelaunay} the orthogonal projection $f:(T,V)\rightarrow P$ bijects $(T,V)$ to a triangulated flat region $(T_0,V_0)$. The orthogonal projection $f$ is then a bi-Lipschitz map satisfying

\begin{equation}\label{eq:Lip}
    \cos(\pi/8)d(x,y) \leq d_0(f(x),f(y)) \leq d(x,y),
\end{equation}
where $d$ is the distance of the cone flat surface $(T,V)$ and $d_0$ is the flat distance of $P$. This follows easily for each triangle of $T$ and is then extended by triangle inequality.

Now, the center $O_1$ of $\Delta$ is in $ABC$ because $\angle ACB \leq \pi/2$ and $AB$ is the longest side. Since $f$ is the identity in $ABC$ and  satisfies (\ref{eq:Lip}), we define the disk $\Delta_1$ in $T_0$ of center $f(O_1)=O_1$ and radius $\cos(\pi/8)\epsilon/2$ does not contain any vertex from $V_0$. the disk $\Delta_1$ is contained in $\DD_r\subseteq T$, so then it is clear that the area of $\Delta_1\cap T$ has a lower bound of the form $K\epsilon^2$ for $K$ constant independent of $\epsilon$.

For the proof of (\ref{2ndclaim}) we proceed the same as in (\ref{1stclaim}) up until the definition of $\Delta_1$, where now we can take $\Delta_1$ to have radius $\cos(\pi/8)2\epsilon/2 = \cos(\pi/8)\epsilon$ by assuming $\ell(\overline{AB})\geq 2\epsilon$. In $(\ref{2ndclaim})$ we cannot assume anymore that $\Delta_1$ is contained in $\DD_r\subseteq T$. Nevertheless, $\Delta_1\cap T$ is obtained by carving out $\Delta_1$ by the segments of $\partial T$, none of which can have their vertices in $\Delta_1$ and all with length less than $\epsilon$. Since the diameter of $\Delta_1$ is $2\cos(\pi/8)\epsilon>\epsilon$, then the smaller disk of radius $\epsilon\sqrt{\cos^2(\pi/8)-1/4}$ and center $O_1$ is contained in $T$.

Observe then that we preserve the $\epsilon$ bound on edges for vertices in $\DD_{r-3\epsilon}$ since they cannot be joined by an edge of length $2\epsilon$ to the exterior of $\DD_{r-\epsilon}$. The potential longer edges could be created in the ring $\DD_{r+\epsilon}\setminus \DD_{r-3\epsilon}$ of width $4\epsilon$.

\textit{Step 3.  How to concatenate various local modifications}\\\noindent One of the problems is that we can gain length in the ring $\DD_{r+\epsilon}\setminus \DD_{r-3\epsilon}$ is such a way that the upper bound piles up until becomes unmanageable. Hence, we need to select how to concatenate local modifications so that the upper bound stays controlled.

Given this information, proceed as follows: at each step take a maximal disjoint family of disk of radius $r$ and perform the Delaunay switches in all of them at the same time. After step one we are left with rings of width $4\epsilon$ where edges lengths are at most $2\epsilon$. Edge lengths are at most $\epsilon$ elsewhere. Pick now a new maximal family of $r$-disks and perform Delaunay switches in all of them at the same time. Note that while in the new rings of $8\epsilon$ width the length is only bounded by $4\epsilon$, in the interior disk the length bound is $\epsilon$. This is because for the proof of (\ref{1stclaim}) we didn't use the bound for edge lengths in $\partial T$. Moreover, because of the proof of (\ref{2ndclaim}), if for nearby edges the length bound is $\eta<2\epsilon$, we can improve the $4\epsilon$ bound to a $2\eta$ bound. Then it is only around the intersections of the $4\epsilon$ widths rings and $8\epsilon$ width rings where the edge length bound is $4\epsilon$, since we can bound by $2\epsilon$ in the other region of the $8\epsilon$ rings. Now for the third family of $r$-disk, we make sure that the region with bound $4\epsilon$ is in the interior where we will get an $\epsilon$ bound. After every step there will be $4\epsilon$ wide rings where the best bound is $2\epsilon$, except for their intersection with $8\epsilon$ wide rings where the best bound is $4\epsilon$. Then we make sure to include this regions in the new interiors for the next family. Since the edge length stays bounded by $4\epsilon$, there is an $r$ that works for that value and assures that we can keep applying this procedure.

Under this process the triangulation stays embedded and strictly decreases $(Area,-Vol)$. Since there are finitely many configurations, we should arrive to a local minimum in finite time. This triangulation will be Delaunay.

In order to complete the proof, we need to show that we have an initial triangulation with small edge lengths. Decompose $\Sigma$ by taking $\mathcal{S}$, the union of some segments between points of $\PP$, such that:

\begin{enumerate}
    \item Each segment in $\mathcal{S}$ has length less than $\epsilon$
    \item\label{interiorcondition} $\mathcal{S}$ can be decomposed into closed curves $\lbrace\gamma_j\rbrace$, such that each closed curve $\gamma_j$ belongs to the $r$ neighbourhood of a point in $\Sigma$.
\end{enumerate}
This is possible since by $\delta$-density we can approximate any path in $\Sigma$ by a polygonal path with vertices in $\PP$ with small edge lengths if $\delta$ is small.

Then using each $\gamma_j$ as relative boundary, we can take the Delaunay triangulation for the points projected to its interior in condition \ref{interiorcondition}. By the discussion in Step 2, the interior edges will be as small as required by taking $\delta$ small. And since we already took care of the relative boundaries, then we produce a triangulation with small edge lengths.

\end{proof}

In the $\mathbb{R}^2$ case, we have that the smallest angle is non-decreasing between Delaunay switches. One can ask the same question for surfaces in $\mathbb{R}^3$. Unfortunately, we have the following counter-example to such claim.

\textbf{Example 1} (Figure \ref{fig:pointy}) Let $ABCD$ be a parallelogram with, $BC$ parallel to $AD$, $|AB|=|CD|$, $\angle ABD > \pi/2$ and the smallest angle is given by $\alpha = \angle BAD = \angle CBD = \angle ACB = \angle ADB$ and no other angles. Observe that in particular $A, B, C, D$ are cyclic. Now, rotate the triangle $ABC$ with respect to $AC$ a small angle (and keeping labels to simplify notation). This modification strictly decreases the length of $BD$, while all other lengths are preserved. Since $\angle ABD > \pi/2$, this implies that the angle $\angle BDA$ will decrease to a value $\beta < \alpha$. In other words, the smallest angle made by the diagonal $BD$ ($\beta$) is strictly smaller than the smallest angle of the diagonal $AC$ ($\alpha$). On top of that, since $BD$ decreased, then the sum of its opposite angles is less than $\pi$. While the sum of opposite angles of $AC$ is $\pi$ (since it has the same angles and lengths as the original quadrilateral $ABCD$), we can do a final modification so that $AC$ is non-Delaunay. This final step is small enough so all angles vary in a controlled way, so that we still have $\beta< \alpha$. For such example, the strict Delaunay diagonal switch replaces $AC$ by $BD$, decreasing the smallest angle from $\alpha$ to $\beta<\alpha$.

\begin{figure}[hbt!]
    \centering
    \tikzset{every picture/.style={line width=0.75pt}} 

\begin{tikzpicture}[x=0.75pt,y=0.75pt,yscale=-1,xscale=1]

\draw   (35.5,124) -- (82.31,58) -- (235.69,58) -- (282.5,124) -- cycle ;
\draw    (346.5,124) -- (591.5,125) -- (546.5,66) -- (411.5,70) -- cycle ;

\draw [color={rgb, 255:red, 74; green, 144; blue, 226 }  ,draw opacity=1 ]   (346.5,124) -- (546.5,66) ;

\draw [color={rgb, 255:red, 208; green, 2; blue, 27 }  ,draw opacity=1 ] [dash pattern={on 0.84pt off 2.51pt}]  (411.5,70) -- (591.5,125) ;

\draw [color={rgb, 255:red, 74; green, 144; blue, 226 }  ,draw opacity=1 ]   (35.5,124) -- (235.69,58) ;

\draw [color={rgb, 255:red, 208; green, 2; blue, 27 }  ,draw opacity=1 ]   (82.31,58) -- (282.5,124) ;

\draw  [draw opacity=0][fill={rgb, 255:red, 155; green, 155; blue, 155 }  ,fill opacity=1 ] (252.52,125) .. controls (252.51,124.67) and (252.5,124.33) .. (252.5,124) .. controls (252.5,120.43) and (253.12,117.01) .. (254.27,113.84) -- (282.5,124) -- cycle ; \draw   (252.52,125) .. controls (252.51,124.67) and (252.5,124.33) .. (252.5,124) .. controls (252.5,120.43) and (253.12,117.01) .. (254.27,113.84) ;
\draw  [draw opacity=0][fill={rgb, 255:red, 155; green, 155; blue, 155 }  ,fill opacity=1 ] (561.5,125) .. controls (561.5,121.8) and (562,118.71) .. (562.93,115.82) -- (591.5,125) -- cycle ; \draw   (561.5,125) .. controls (561.5,121.8) and (562,118.71) .. (562.93,115.82) ;

\draw (30,127) node   [align=left] {A};
\draw (338,125) node   [align=left] {A};
\draw (70,52) node   [align=left] {B};
\draw (401,62) node   [align=left] {B};
\draw (246,53) node   [align=left] {C};
\draw (562,57) node   [align=left] {C};
\draw (291,124) node   [align=left] {D};
\draw (603,124) node   [align=left] {D};
\draw (236,118) node   [align=left] {$\displaystyle \alpha $};
\draw (539,118) node   [align=left] {$\displaystyle \beta $};

\end{tikzpicture}
    \caption{}
    \label{fig:pointy}
\end{figure}\noindent
Notice that since the perturbations where small and $A, B, C, D$ were coplanar, we have counter-examples for configurations as close to planar as we want. Hence the step decrease of the smallest angle cannot be ruled out by having a fine approximation of the surface alone.

Finally, and as another counterpoint, let us see an example where a Delaunay triangulation produces an arbitrary large proportion of very-thin triangles. 

\textbf{Example 2} Note first that the hyperbolic convex hull $\CC$ of a triangle $ABC\subset\mathbb{R}^2\subset\partial\mathbb{H}^3$ is described as follows. Let $A', B', C'$ be the points of tangency of the incircle with the triangle $ABC$. Above the triangular region $A'B'C'$, $\partial\CC$ coincides with the hyperplane with boundary equal the incircle. In the remaining region, $\partial\CC$ is formed by the geodesics lines whose projection to $\mathbb{R}^2$ are orthogonal to the nearest bisector. Then $\CC$ is the upper region determined by $\partial \CC$.

For a triangulation $T$, take $ABC$ a triangle on it and the tangency points $A', B', C'$ of the incircle, and denote by $W$ the set of points $A', B', C'$ while considering each possible triangle in $T$. Fix $3\epsilon$ a lower bound for the distances between points of $W$ and points in $V$, the vertex set of $T$. Given $n\gg0$, for each edge $AB$ in $T$ take the two points $X, Y$ at distance $\epsilon$ from either $A$ or $B$ and subdivide $XY\subset AB$ in $n$ segments of equal length, and name $\overline{W}$ the set of points obtained while doing this for each edge in $T$. In each triangle $ABC$ of $T$ take the Delaunay triangulation of $A, B, C,\overline{W}\cap ABC$. This triangulation will have an interior triangle approximating the triangle $A'B'C'$, very-thin triangle with arbitrarily small angle facing and edge of $ABC$ and isosceles corners. This follows from the hyperbolic convex hull set described previously. While assembling the triangles of $T$ we see that opposite angles are all acute, so they satisfy the strict Delaunay condition. Nevertheless, the proportion of thin triangles goes to $1$ as $n\rightarrow\infty$. This could be solved by considering a point cloud that equidistributes.

\bibliographystyle{amsalpha}
\bibliography{mybib}
\end{document}